\documentclass[compsoc, a4paper, onecolumn, 11pt]{IEEEtran}

\usepackage[margin=1in]{geometry}
\usepackage{amsmath,amssymb,amsfonts,amsthm,mathtools}
\usepackage{bm,tikz}
\usepackage{algpseudocode}
\usepackage{algorithm}
\usepackage{algorithmicx}
\usepackage{graphicx}
\usepackage{textcomp}
\usepackage{xcolor,soul}
\usepackage{subfigure}
\usepackage{url,hyperref}
\usepackage{authblk}
\usepackage{setspace}
\usepackage{apacite}
\usepackage{natbib}
\usepackage{enumerate,paralist}

\newtheorem{theorem}{Theorem}
\newtheorem{conjecture}{Conjecture}

\newcommand{\R}{\mathbb{R}}
\newcommand{\G}{\mathcal{G}}
\newcommand{\algoname}[1]{\textnormal{\textsc{#1}}}

\DeclareMathOperator{\Hc}{Hc}
\DeclareMathOperator{\CPA}{CPA}
\DeclareMathOperator{\CPD}{CPD}

\algnewcommand\algorithmicforeach{\textbf{for each}}
\algdef{S}[FOR]{ForEach}[1]{\algorithmicforeach\ #1\ \algorithmicdo}
\algnewcommand\AND{\textbf{ and }}

\author[*]{Valentino Vito}
\author[ ]{Lim Yohanes Stefanus}
\affil[ ]{Faculty of Computer Science, Universitas Indonesia, Depok 16424, Indonesia}
\affil[ ]{E-mail addresses: valentino.vito11@ui.ac.id, yohanes@cs.ui.ac.id}
\affil[*]{Corresponding author}

\begin{document}

\title{Adaptive Monte Carlo search for conjecture refutation in graph theory}
\maketitle

\setstretch{1.5}

\begin{abstract}
Graph theory is an interdisciplinary field of study that has various applications in mathematical modeling and computer science. Research in graph theory depends on the creation of not only theorems but also conjectures. Conjecture-refuting algorithms attempt to refute conjectures by searching for counterexamples to those conjectures, often by maximizing certain score functions on graphs. This study proposes a novel conjecture-refuting algorithm, referred to as the adaptive Monte Carlo search (AMCS) algorithm, obtained by modifying the Monte Carlo tree search algorithm. Evaluated based on its success in finding counterexamples to several graph theory conjectures, AMCS outperforms existing conjecture-refuting algorithms. The algorithm is further utilized to refute six open conjectures, two of which were chemical graph theory conjectures formulated by Liu et al.\ in 2021 and four of which were formulated by the AutoGraphiX computer system in 2006. Finally, four of the open conjectures are strongly refuted by generalizing the counterexamples obtained by AMCS to produce a family of counterexamples. It is expected that the algorithm can help researchers test graph-theoretic conjectures more effectively.
\end{abstract}


\section{Introduction}
Graph theory is a fast-growing field of study with numerous applications in modeling various objects including computer networks, semantic networks, graphical models, and molecular compounds. Research in graph theory highly depends on the creation of new theorems and conjectures. Unlike theorems, conjectures are mathematical statements that have not been proved to be either true or false \citep{gera}. The resolution of conjectures may bring forth new theorems and advance the development of the field of study. In addition, conjecture-making promotes collaboration between mathematicians, so it is a vital aspect of mathematical research.

Two subfields of graph theory with an abundance of conjectures are spectral graph theory and chemical graph theory. In spectral graph theory, matrices associated with graphs, along with their eigenvalues, are linked to the structural properties of graphs \citep{nica}. In chemical graph theory, molecular structures of chemical compounds are analyzed via graphs \citep{wagner2018}. Chemical graph theory allows the analysis of chemical compounds to be done outside the laboratory, with computations performed using a computer instead of relying on physical samples. As such, chemical graph theory plays a role in the advancement of chemical research, especially in drug discovery \citep{estrada}.

Various computer systems have been developed for spectral and chemical graph theory research since the late 20th century. \citet{cvetkovic} created an expert system for graph theory known as GRAPH. It is a useful tool to assist in computing numerous graph-theoretical properties. Using a database of known graphs, \citet{fajtlowicz} developed a conjecture-finding computer program called Graffiti, which searches for conjectures whose statements involve an inequality. More recently, \citet{hansen2000} created AutoGraphiX, an automated system that can be used to discover new conjectures using the variable neighborhood search metaheuristic \citep{hansen2001}.

In contrast to conjecture-finding algorithms, conjecture-refuting algorithms attempt to refute known conjectures by searching for counterexamples. The refutation of conjectures is beneficial to narrow down true conjectures. \citet{wagner2021} applied the deep cross-entropy method, a reinforcement learning algorithm, to refute several conjectures in combinatorics and graph theory. Improving on this work, \citet{roucairol} applied algorithms based on Monte Carlo tree search to refute conjectures in spectral graph theory in a much faster way than with Wagner's method, namely by the nested Monte Carlo search (NMCS) and the nested rollout policy adaptation (NRPA) algorithms.

This study further develops NMCS to be more successful in searching for graph counterexamples to graph theory conjectures. The conjectures considered in this study specifically belong to spectral and chemical graph theory. This is due to the large number of tractable conjectures that are suitable for computation within the two subfields. The proposed algorithm of this study is referred to as the \textit{adaptive Monte Carlo search} (AMCS) algorithm. The algorithm is evaluated based on its success in finding a counterexample when applied to various conjectures, which consist of four already resolved conjectures and six currently open conjectures. Two of the open conjectures are chemical graph theory conjectures formulated by \citet{liu}, whereas the rest are formulated via the AutoGraphiX system. It is expected that the proposed algorithm, whose code is publicly accessible\footnote{Source code: \url{https://github.com/valentinovito/Adaptive_MC_Search}}, can help researchers test conjectures in a more effective and systematic manner. Moreover, based on structural patterns present within AMCS's counterexamples, a generalization of the counterexamples by way of obtaining a family of counterexamples is provided for four of the open conjectures.

\section{Research questions}
This research is motivated by the limitations of existing algorithms in finding counterexamples for various types of conjectures. The state-of-the-art Monte Carlo search algorithms applied by \citet{roucairol}, namely NMCS and NRPA, were unable to refute every spectral graph theory conjecture that was tested against them. Therefore, there is a need to develop a more successful algorithm that is able to refute conjectures in which previous algorithms fail.

\begin{center}
\begin{enumerate}[\textbf{RQ1.}]
    \item Is there a more successful algorithm that is able to refute more conjectures than existing algorithms? 
\end{enumerate}
\end{center}

To address RQ1, we propose a novel Monte Carlo search algorithm: the adaptive Monte Carlo search (AMCS). AMCS is designed to handle various types of conjectures beyond spectral graph theory. The performance of the proposed algorithm is measured by the number of conjectures it successfully refutes. A conjecture is successfully refuted if a counterexample is found within a reasonable time limit, which is set to $12$ hours in this study. Hence, AMCS is said to be more successful if it is able to refute more conjectures than NMCS and NRPA, where the set of conjectures under consideration consists of four conjectures that have previously been tested on NMCS and NRPA.

The set of four conjectures tested by \citet{roucairol} consists only of resolved conjectures in spectral graph theory which are already known to be false. To further test the versatility of the proposed algorithm, additional tests are performed using six open conjectures in both spectral and chemical graph theory. To our knowledge, the correctness of these open conjectures is not yet known previously.

\begin{center}
\begin{enumerate}[\textbf{RQ2.}]
    \item Is the proposed algorithm suitable for use on currently open conjectures?
\end{enumerate}
\end{center}

RQ2 is addressed by finding open conjectures from graph theory papers that can potentially be refuted. Since these conjectures are still open, there is a possibility that the algorithm fails to find a counterexample due to these conjectures being true. But if the algorithm manages to refute several open conjectures, then there is a strong indication that the algorithm is suitable for other open conjectures in the literature.

\section{Literature review}
In this section, a literature study on relevant subjects, including preliminaries in graph theory and search algorithms applicable to graph theory, is provided.

\subsection{Graph-theoretic preliminaries}
Formally, a \textit{graph} is a pair $G = (V, E)$ of sets such that the elements of $E$ are $2$-element subsets of $V$ \citep{diestel}. The elements of $V$ are called the \textit{vertices} of $G$, and the elements of $E$ are called the \textit{edges} of $G$. An edge $\{u, v\}$ is often written as $uv$ for brevity. The \textit{vertex set} $V$ and \textit{edge set} $E$ of a $G$ can also be written as $V(G)$ and $E(G)$, respectively. Given a graph $G = (V, E)$, its \textit{complement} $\overline{G}$ is a graph whose vertex set is $V$ and whose edges are the pairs of nonadjacent vertices of $G$ \citep{bondy}. The \textit{order} of $G$ is the number of its vertices, whereas its \textit{size} is the number of its edges. Vertices $u, v \in V(G)$ are \textit{adjacent} if $uv \in E(G)$. If any two vertices in $G$ are adjacent, then $G$ is \textit{complete} and denoted by $K_n$. The \textit{degree} of $v$, denoted by $d_v$, is the number of vertices adjacent to $v$.

The \textit{path} $P_n = v_1v_2\dots v_n$, where the $v_i$'s are all distinct, is a graph defined by $V(P_n) = \{v_1, v_2, \dots, v_n\}$ and $E(P_n) = \{v_1v_2, v_2v_3, \dots, v_{n-1}v_n\}$. Here, $v_1$ and $v_n$ form the \textit{endpoints} of $P_n$. The \textit{cycle} $C_n$, $n \ge 3$, is obtained by adding the edge $v_1v_n$ to $P_n$. The \textit{star} $S_n$ is a graph of order $n$ such that one of its vertices (called its \textit{center}) has degree $n - 1$, while the rest have degree $1$. The \textit{distance} $d(u, v)$ of two vertices $u, v \in V(G)$ is the length of the shortest path from $u$ to $v$. A graph $G$ is \textit{connected} if any two of its vertices are endpoints of some path in $G$. A graph $G$ is \textit{acyclic} if it does not contain any cycles. A connected, acyclic graph is called a \textit{tree}, which is often denoted by $T$ instead of $G$. A vertex of degree $1$ is often called a \textit{leaf}. A connected graph on $n$ vertices is a tree if and only if it has $n - 1$ edges \citep{diestel}.

\subsection{Spectral graph theory}
Spectral graph theory studies the properties of graphs through matrices associated with them. The adjacency matrix is an example of such a matrix. The \textit{adjacency matrix} $A$ of a graph $G$ is a square matrix indexed by $V(G)$ such that $A(u, v) = 1$ whenever $u$ is adjacent to $v$ and $A(u, v) = 0$ otherwise. \citet{graham} defined the \textit{distance matrix} $D$ of $G$ as a square matrix indexed by $V(G)$ such that $D(u, v) = d(u, v)$, where $d$ is the distance function. The adjacency and distance matrices are symmetric, so their eigenvalues are all real numbers \citep{godsil}. The set of eigenvalues $\lambda_1 \ge \lambda_2 \ge \dots \ge \lambda_n$ of an $n \times n$ adjacency matrix, repeated according to their multiplicity, is called the \textit{adjacency spectrum}. Likewise, the set of eigenvalues $\partial_1 \ge \partial_2 \ge \dots \ge \partial_n$ of an $n \times n$ distance matrix is called the \textit{distance spectrum}. The characteristic polynomials of the adjacency matrix and the distance matrix are denoted by $\CPA(G)$ and $\CPD(G)$, respectively. Another matrix associated with graphs is the \textit{Laplacian matrix} $L$, which is a square matrix indexed by $V(G)$ such that $L(u, v) = d_v$ if $u = v$, $L(u, v) = -1$ if $u$ is adjacent to $v$, and $L(u, v) = 0$ otherwise \citep{nica}.

\subsection{Chemical graph theory}
Chemical graph theory is a field of study that analyzes the molecular structure of chemical compounds through graphs \citep{wagner2018}. One aim of chemical graph theory is to study the qualities of various compounds by discovering the structural properties of the graphs representing them. In order to quantify the structural properties of graphs, certain measures (or \textit{invariants}) referred to as \textit{topological indices} are defined on these graphs. A topological index provides a summary of the topological structure of a graph in the form of a numerical quantity.

A well-known example of a topological index is the \textit{Randi\'c index} introduced by \citet{randic}. Instead of being defined as the sum of distances, the Randi\'c index $R(G)$ is defined based on degrees:
\[R(G) = \sum_{uv \in E(G)} (d_u d_v)^{-\frac{1}{2}} = \sum_{uv \in E(G)} \frac{1}{\sqrt{d_u d_v}}.\]
A more general version of the Randi\'c index, known as the \textit{generalized Randi\'c index} $R_\alpha(G)$, is given by the formula
\[R_\alpha(G) = \sum_{uv \in E(G)} (d_u d_v)^\alpha.\]
In particular, the index $R_1(G)$ is sometimes referred to as the \textit{second Zagreb index} $M_2(G)$ \citep{wagner2018}.

\subsection{Conjecture-related algorithms}
This subsection presents several previous studies on algorithms concerning mathematical conjectures, with an emphasis on conjectures related to graph theory. The two types of algorithms discussed here are conjecture-finding algorithms and conjecture-refuting algorithms.

\subsubsection{Conjecture-finding algorithms}
An early computer program designed to discover new conjectures was written in 1984 and was known by the name Graffiti \citep{fajtlowicz}. The program tests potential graph theory conjectures using a database of known graphs. If none of the graphs in the database is a counterexample to the conjecture, then the conjecture is proposed as a potential theorem. Graffiti was applied by \citet{fajtlowicz} to conjectures of the form $I \le J$, $I \le J + K$, or $I + J \le K + L$, where $I$, $J$, $K$, and $L$ are either graph invariants (such as the Randi\'c index) or equal to the constant $1$.

A more modern computer system, known as the AutoGraphiX system \citep{hansen2000}, employed a more sophisticated algorithm to find conjectures. AutoGraphiX specifically tackles extremal problems in graph theory, thereby providing a connection between combinatorial optimization and extremal graph theory \citep{aouchiche2008}. Given connected graphs $G$ with $n = \lvert V(G) \rvert$, AutoGraphiX creates conjectures of the form $\ell_n \le I \oplus J \le u_n$, where $I$ and $J$ are invariants, $\oplus$ is one of the operations in the set $\{+,-,\times,/\}$, and $\ell_n$ and $u_n$ are lower and upper bounds, respectively, of $I \oplus J$ expressed as a function of $n$ \citep{aouchiche2010}. AutoGraphiX first applies the variable neighborhood search metaheuristic to find, for every $n$, an \textit{extremal graph} $G_n$ (that is, a graph $G_n$ such that the numerical value of $I \oplus J$ is either maximized or minimized over all graphs of order $n$). Afterward, $\ell_n$ is set to $I(G_n) \oplus J(G_n)$ if $G_n$ minimizes $I \oplus J$, and $u_n$ is set to $I(G_n) \oplus J(G_n)$ if $G_n$ maximizes $I \oplus J$. 

\subsubsection{Conjecture-refuting algorithms}
Numerous conjectures in graph theory can be refuted by providing graph counterexamples to said conjectures. However, exhaustive search is usually not a feasible approach, even for small graphs. For illustration, there are over a billion distinct connected graphs on $11$ vertices alone \citep{sloane}. As a result, a more effective search algorithm is usually necessary to obtain a counterexample within a reasonable amount of time. In practice, this combinatorial search problem can be reformulated as a combinatorial optimization problem, where a score function associated with the conjecture is maximized to obtain the desired counterexample.

\citet{wagner2021} applied a deep reinforcement learning algorithm called the deep cross-entropy method for the conjecture refutation task. The method uses a neural network to construct a number of graphs each iteration and keeps only the top-scoring constructions. Based on the constructions that are kept, the weights of the neural network are then adjusted to minimize the cross-entropy loss. As a result, the neural network learns from the best-scoring graph constructions so that these graphs are more likely to be constructed during the next iteration. This method was also able to be modified to be used on other combinatorial conjectures other than graph theory.

Wagner's method managed to refute an AutoGraphiX conjecture on the largest eigenvalue of the adjacency matrix of a graph. However, it took $5000$ iterations to obtain the graph that can act as its counterexample. To improve the search time of the algorithm, \citet{roucairol} employed Monte Carlo search algorithms on the same task. Using the nested Monte Carlo search (NMCS) and the nested rollout policy adaptation (NRPA) algorithms, the same AutoGraphiX conjecture, along with other conjectures, was able to be refuted within seconds. This effectively makes NMCS and NRPA the state-of-the-art algorithms for the conjecture refutation task in graph theory.

\subsection{Monte Carlo search}
Monte Carlo tree search (MCTS) is a probabilistic algorithm for finding optimal decisions by taking random samples in the decision space \citep{browne}. Its primary application is in constructing artificial intelligence for playing games, for which MCTS searches for optimal decisions in certain game trees. Numerous variations of MCTS have been designed to handle specific tasks. In particular, \citet{cazenave} introduced a variant of MCTS known as the nested Monte Carlo search (NMCS).

NMCS is a recursive algorithm that requires two parameters: \textit{depth} and \textit{level}. Essentially, the algorithm recursively calls lower-level NMCS on child states of the current state in order to decide which move to play next \citep{roucairol}. Child states refer to states that can be obtained from the current state by performing any one move. On each child state, NMCS calls itself, but its level is decremented to $\textit{level} - 1$ for the call. This is precisely the nested feature that distinguishes NMCS from ordinary MCTS. The recursive property of NMCS further increases the computational complexity of the algorithm in exchange for a significantly more thorough search. At \textit{level} = 0, NMCS performs $k$ random moves, where $k = \textit{depth}$. The resulting state is designated as the terminal state associated with the current child state. This terminal state will become the new best state if it improves the score of the current best state.

Nested rollout policy adaptation (NRPA) is a modification of NMCS. In essence, NRPA is a combination of NMCS and Q-learning. The algorithm uses a policy that attributes a value to a move \citep{roucairol}. As with NMCS, the NRPA algorithm is a recursive algorithm that calls lower-level NRPA to decide which move to play next. But at $\textit{level} = 0$, NRPA performs random moves using a softmax distribution based on its continually updated policy.

\begin{algorithm}
\caption{Nested Monte Carlo search for the space of trees.}\label{alg:nmcs_trees}
\begin{algorithmic}[1]
\Require{A score function $s\colon \G \to \R$ to be maximized.}
\Function{NMCSTrees}{$G$, $\textit{depth}$, $\textit{level}$}
    \State $\textit{bestGraph} \gets G$;
    \State $\textit{bestScore} \gets s(G)$;
    \If{\textit{level} = 0}
        \State $G' \gets G$;
        \For{$i = 1$ to $\textit{depth}$}
            \State Add a random leaf or subdivision to $G'$;
        \EndFor
        \If{$s(G') > \textit{bestScore}$}
            \State $\textit{bestGraph} \gets G'$;
        \EndIf
    \Else
        \ForEach{$x \in V(G) \cup E(G)$}
            \State $G' \gets G$;
            \If{$x$ is a vertex}
                \State Add a leaf to $G'$ on the vertex $x$;
            \Else
                \State Add a subdivision to $G'$ on the edge $x$;
            \EndIf
            \State $G' \gets$ \Call{NMCSTrees}{$G'$, $\textit{depth}$, $\textit{level} - 1$};
            \If{$s(G') > \textit{bestScore}$}
                \State $\textit{bestGraph} \gets G'$; $\textit{bestScore} \gets s(G')$;
            \EndIf
        \EndFor
    \EndIf
    \State \Return $\textit{bestGraph}$;
\EndFunction
\end{algorithmic}
\end{algorithm}

\begin{algorithm}
\caption{Nested Monte Carlo search for the space of connected graphs.}\label{alg:nmcs_conngraphs}
\begin{algorithmic}[1]
\Require{A score function $s\colon \G \to \R$ to be maximized.}
\Function{NMCSConnectedGraphs}{$G$, $\textit{depth}$, $\textit{level}$}
    \State $\textit{bestGraph} \gets G$;
    \State $\textit{bestScore} \gets s(G)$;
    \If{\textit{level} = 0}
        \State $G' \gets G$;
        \For{$i = 1$ to $\textit{depth}$}
            \State Add a random leaf, subdivision, or edge to $G'$;
        \EndFor
        \If{$s(G') > \textit{bestScore}$}
            \State $\textit{bestGraph} \gets G'$;
        \EndIf
    \Else
        \ForEach{$x \in V(G) \cup E(G) \cup E(\overline{G})$}
            \State $G' \gets G$;
            \If{$x$ is a vertex}
                \State Add a leaf to $G'$ on the vertex $x$;
            \ElsIf{$x$ is an edge in $G$}
                \State Add a subdivision to $G'$ on the edge $x$;
            \Else
                \State Add the edge $x$ to $G'$;
            \EndIf
            \State $G' \gets$ \Call{NMCSConnectedGraphs}{$G'$, $\textit{depth}$, $\textit{level} - 1$};
            \If{$s(G') > \textit{bestScore}$}
                \State $\textit{bestGraph} \gets G'$; $\textit{bestScore} \gets s(G')$;
            \EndIf
        \EndFor
    \EndIf
    \State \Return $\textit{bestGraph}$;
\EndFunction
\end{algorithmic}
\end{algorithm}

NMCS can be applied to graph optimization problems of the form $\max\{s(G): G \in \G\}$, where $\G$ denotes the feasible set of graphs and $s(G)$ denotes a real-valued objective function referred to as a \textit{score function}. The score function $s(G)$ is chosen such that the conjecture is refuted by a graph $G \in \G$ as its counterexample if $s(G) > 0$. The feasible set $\G$ is designated as the search space of the algorithm, and it depends on the hypotheses of the conjecture under consideration. In this study, $\G$ is restricted to either the set of connected graphs or the set of trees with sufficiently many vertices.

Two NMCS algorithms for the aforementioned graph optimization problem are provided in Algorithms \ref{alg:nmcs_trees} (when the search space $\G$ consists of trees) and \ref{alg:nmcs_conngraphs} (when the search space $\G$ consists of connected graphs). Given a graph $G \in \G$, both algorithms include two possible moves that can be carried out to expand $G$ in order to explore the search space: adding a leaf to some vertex or performing an edge subdivision. To \textit{subdivide} an edge $uv$ is to delete $uv$, add a new vertex $w$, and join $w$ to both $u$ and $v$ \citep{bondy}.

In essence, Algorithm \ref{alg:nmcs_trees} attempts to discover which combination of the aforementioned two moves is likely to produce a terminal graph that maximizes the score function. Since adding a leaf or subdivision to a tree produces another tree, Algorithm \ref{alg:nmcs_trees} is guaranteed to only search within the class of trees as long as the initial graph $G$ is a tree. Algorithm \ref{alg:nmcs_conngraphs} differs from Algorithm \ref{alg:nmcs_trees} in that it includes a third possible move to expand a graph, which is to add an edge to a pair of existing vertices. This move cannot be included in Algorithm \ref{alg:nmcs_trees} since the move may produce a non-tree graph. However, the addition of an extra edge cannot disconnect a connected graph. Hence, as long as the initial graph $G$ is connected, Algorithm \ref{alg:nmcs_conngraphs} only operates within the search space of connected graphs. Having this third possible move increases the complexity of the algorithm, but should be utilized if Algorithm \ref{alg:nmcs_trees} fails to produce a counterexample.

\section{Conjectures for experimentation}

There are a total of ten conjectures that are used for the experiments, and they are stated in Conjectures \ref{conj_1}--\ref{conj_10}. The conjectures are classified into two groups: (1) four resolved conjectures previously refuted by another algorithm, and (2) six open conjectures which have never been refuted at the time of this writing. The resolved conjectures are employed to compare the proposed algorithm to state-of-the-art conjecture-refuting algorithms by \citet{roucairol} (addressing RQ1), while the open conjectures are employed to expose the proposed algorithm to as yet unsolved problems in the field (addressing RQ2).

\subsection{Resolved conjectures}

For the experiments, four conjectures are used to evaluate the performance of the proposed algorithm. All four conjectures lie in spectral graph theory. These conjectures were previously applied by \citet{roucairol} to test their Monte Carlo algorithms. Before the conjectures can be stated, several graph-theoretical notions need to be provided.

The \textit{diameter} $D(G)$ of a graph $G$ is the maximum possible distance $d(u, v)$ between two vertices $u, v \in V(G)$ \citep{aouchiche2016}, namely
\[D(G) = \max_{\{u, v\} \subseteq V(G)} d(u, v).\]
On the other hand, the \textit{proximity} $\pi(G)$ of a graph $G$ is the minimum average distance from a vertex of $G$ to all the other vertices \citep{aouchiche2016}, namely
\[\pi(G) = \min_{u \in V(G)} \frac{\sum_{v \in V(G)} d(u, v)}{n - 1}.\]
The largest eigenvalue and second largest eigenvalue of the adjacency matrix of $G$ are denoted by $\lambda_1(G)$ and $\lambda_2(G)$, respectively. A \textit{matching} in a graph $G$ is a set of pairwise nonadjacent edges of $G$, and the \textit{matching number} $\mu(G)$ of $G$ is the maximum possible size of a matching in $G$ \citep{bondy}.

The first two resolved conjectures, both proposed by Aouchiche and Hansen, discuss bounds involving the adjacency and distance spectrums of $G$.

\begin{conjecture}[\citealp{aouchiche2010}]\label{conj_1}
Let $G$ be a connected graph on $n \ge 3$ vertices. Then
\begin{equation}\label{eq_1}
\lambda_1(G) + \mu(G) \ge \sqrt{n - 1} + 1.
\end{equation}
\end{conjecture}

\begin{conjecture}[\citealp{aouchiche2016}]\label{conj_2}
Let $G$ be a connected graph on $n \ge 4$ vertices with distance spectrum $\partial_1 \ge \partial_2 \ge \dots \ge \partial_n$. Then
\begin{equation}\label{eq_2}
\pi(G) + \partial_{\left\lfloor\frac{2}{3}D(G)\right\rfloor} > 0.
\end{equation}
\end{conjecture}

Let $A$ and $D$ denote the adjacency and distance matrix of a tree $T$, respectively. Also, let $CPA(G) = \sum_{i = 0}^m a_ix^{\alpha_i}$ (where $\alpha_0 < \alpha_2 < \dots < \alpha_m$ and $a_i \neq 0$ for all $i$) and $CPD(G) = \sum_{i = 0}^n \delta_ix^i$ be the characteristic polynomials of $A$ and $D$, respectively. The third resolved conjecture is about the position of the peak coefficient of those two characteristic polynomials. The position of the peak of the absolute values of the non-zero coefficients $\lvert a_0 \rvert, \dots, \lvert a_m \rvert$ of $CPA(T)$ is denoted by $p_A(T)$, while the number of non-zero coefficients of $CPA(T)$ is denoted by $m(T)$. For $0 \le i \le n - 2$, let $d_i = \frac{2^i}{2^{n-2}} \lvert\delta_i\rvert$ be the \textit{normalized coefficients} of $CPD(T)$. The position of the peak of the normalized coefficients $d_0, \dots, d_{n - 2}$ of $CPD(T)$ is denoted by $p_D(T)$.

\begin{conjecture}[\citealp{collins}]\label{conj_3}
The peak of the sequence of absolute values of the non-zero coefficients of $CPA(T)$ (that is, the peak of $(\lvert a_i \rvert)_{0 \le i \le m}$) is relatively located at the same place as the peak of the sequence of normalized coefficients of $CPD(T)$ (that is, the peak of $(d_i)_{1 \le i \le n - 2}$) when the locations of the peaks are counted from opposite ends of the two sequences. In other words,
\begin{equation}\label{eq_3}
\frac{p_A(T)}{m(T)} = 1 - \frac{p_D(T)}{n}.
\end{equation}
\end{conjecture}

The last resolved conjecture relates to the \textit{harmonic} of a graph $G$, which is denoted by $\Hc(G)$ \citep{favaron} and defined by:
\[\Hc(G) = \sum_{uv \in E(G)} \frac{2}{d_u + d_v}.\]

\begin{conjecture}[\citealp{favaron}]\label{conj_4}
For any graph $G$,
\begin{equation}\label{eq_4}
\lambda_2(G) \le \Hc(G).
\end{equation}
\end{conjecture}

\subsection{Open conjectures}

The first two open conjectures by \citet{liu} discuss bounds on a modified version of the second Zagreb index $M_2(G)$ known as the \textit{modified second Zagreb index} $\prescript{m}{}M_2(G)$. This modification was introduced by \citet{nikolic}. The index $\prescript{m}{}M_2(G)$ is a special case of the generalized Randi\'c index when $\alpha = -1$, namely $\prescript{m}{}M_2(G) = R_{-1}(G)$. Hence
\[\prescript{m}{}M_2(G) = \sum_{uv \in E(G)} \frac{1}{d_u d_v}.\]

\begin{conjecture}[\citealp{liu}]\label{conj_5}
Let $T$ be a tree on $n$ vertices. Then
\begin{equation}\label{eq_5}
\prescript{m}{}M_2(T) \le \frac{n + 1}{4}.
\end{equation}
\end{conjecture}

A \textit{dominating set} in a graph $G$ is a set $S \subseteq V(G)$ of vertices such that every vertex $v \in V(G)$ is either an element of $S$ or adjacent to an element of $S$ \citep{haynes}. The \textit{domination number} $\gamma(G)$ of $G$ is the minimum possible size of a dominating set in $G$.

\begin{conjecture}[\citealp{liu}]\label{conj_6}
Let $T$ be a tree on $n$ vertices with domination number $\gamma$. Then
\begin{equation}\label{eq_6}
\prescript{m}{}M_2(T) \ge - \frac{\gamma(T) - 1}{2 (n - \gamma(T))} + \frac{\gamma(T) + 1}{2}.
\end{equation}
\end{conjecture}

The last four conjectures were conjectures left open by the computer system AutoGraphiX. To the best of our knowledge, these conjectures have yet to be refuted since their initial conception in 2006.

\begin{conjecture}[\citealp{aouchiche2006}]\label{conj_7}
Let $G$ be a connected graph on $n \ge 3$ vertices. Then
\begin{equation}\label{eq_7}
\lambda_1(G) \cdot \pi(G) \le n - 1.
\end{equation}
\end{conjecture}

The following conjecture requires the second smallest eigenvalue of the Laplacian matrix of $G$, which is denoted by $a(G)$ and known as the \textit{algebraic connectivity of $G$}.

\begin{conjecture}[\citealp{aouchiche2006}]\label{conj_8}
Let $G$ be a connected graph on $n \ge 3$ vertices. Then
\begin{equation}\label{eq_8}
a(G) \cdot \pi(G) \ge \begin{cases}
			\frac{n ^ 2}{2(n - 1)} \left(1 - \cos\frac{\pi}{n}\right) & \text{if $n$ is even},\\
            \frac{n + 1}{2} \left(1 - \cos\frac{\pi}{n}\right) & \text{otherwise}.
		 \end{cases}
\end{equation}
\end{conjecture}

An \textit{independent set} $I$ in a graph $G$ is a set of vertices of which no two are adjacent \citep{bondy}. This independent set $I$ is \textit{maximum} if $G$ does not contain a larger independent set. The number of vertices in a maximum independent set is called the \textit{independence number of $G$} and is denoted by $\alpha(G)$.

\begin{conjecture}[\citealp{aouchiche2006}]\label{conj_9}
Let $G$ be a connected graph on $n \ge 3$ vertices. Then
\begin{equation}\label{eq_9}
\lambda_1(G) - \alpha(G) \ge \sqrt{n - 1} - n + 1.
\end{equation}
\end{conjecture}

The final open AutoGraphiX conjecture concerns the Randi\'c index $R(G)$ of $G$.

\begin{conjecture}[\citealp{aouchiche2006}]\label{conj_10}
Let $G$ be a connected graph on $n \ge 3$ vertices. Then
\begin{equation}\label{eq_10}
R(G) + \alpha(G) \le n - 1 + \sqrt{n - 1}.
\end{equation}
\end{conjecture}

\section{Proposed algorithm: adaptive Monte Carlo search}\label{sub:amcs}

The conjecture-refuting algorithm proposed in this study is a novel modification of NMCS referred to as the \textit{adaptive Monte Carlo search} (AMCS). One limitation of NMCS is that its search is restricted to graphs of a certain size. For example, if $G$ has $m$ edges, then NMCS consistently checks graphs $G'$ of size $m + \textit{level} + \textit{depth}$. If NMCS does not find a more optimal graph $G'$ such that $s(G') > s(G)$, the algorithm will be stuck at searching graphs of size $m + \textit{level} + \textit{depth}$. This is not an ideal scenario since the algorithm will be unable to explore a large variety of graphs when it is stuck at a local maximum. As such, NMCS often stops progressing when it cannot find any better graph than the current one.

\begin{algorithm}
\caption{Adaptive Monte Carlo search (AMCS).}\label{alg:amcs}
\begin{algorithmic}[1]
\Require{A score function $s\colon \G \to \R$ to be maximized.}
\Function{AMCS}{$G$, $\textit{maxDepth}$, $\textit{maxLevel}$, $\textit{treesOnly}$}
    \If{$\textit{treesOnly} = \textbf{true}$}
        \State $\algoname{NMCS} \gets \algoname{NMCSTrees}$;
    \Else
        \State $\algoname{NMCS} \gets \algoname{NMCSConnectedGraphs}$;
    \EndIf
    \State $\textit{depth} \gets 0$; $\textit{level} \gets 1$; $\textit{minOrder} \gets \lvert V(G) \rvert$;
    \While{$s(G) \le 0$ \AND $\textit{level} \le \textit{maxLevel}$}
        \State $G' \gets G$;
        \While{$\lvert V(G') \rvert > \textit{minOrder}$}
            \State $\textit{randReal} \gets$ random real number between $0$ and $1$;
            \If{$\textit{randReal} < \textit{depth} / (\textit{depth} + 1)$}
                \State Remove a random leaf or subdivision from $G'$;
            \Else
                \State Break the while loop;
            \EndIf
        \EndWhile
    \State $G' \gets$ \Call{NMCS}{$G'$, $\textit{depth}$, $\textit{level}$};
    \If{$s(G') > s(G)$}
        \State $G \gets G'$; $\textit{depth} \gets 0$; $\textit{level} \gets 1$;
    \ElsIf{$\textit{depth} < \textit{maxDepth}$}
        \State $\textit{depth} \gets \textit{depth} + 1$;
    \Else
        \State $\textit{depth} \gets 0$; $\textit{level} \gets \textit{level} + 1$;
    \EndIf
    \EndWhile
    \State \Return $G$;
\EndFunction
\end{algorithmic}
\end{algorithm}

To address this issue, AMCS tries to adapt to the situation by increasing the \textit{depth} and \textit{level} parameters if it detects that it is stuck at a local maximum. Instead of having \textit{depth} and \textit{level} as preset parameters, AMCS increments the values of both parameters if a more optimal graph is not found during the last iteration. Consequently, AMCS attempts to slightly adjust the current best graph to steadily increase its score, before adjusting the graph in a more dramatic fashion if a better graph is not found. This concept is based on the idea of neighborhood change exploited by variable neighborhood search. As in variable neighborhood search, the AMCS algorithm continually changes the search depth so that it may adapt to the current difficulty of the search problem. In other words, when the exploration is stuck at a local maximum, AMCS increases the \textit{depth} and \textit{level} so that the search procedure becomes more complex. The aim of this feature is to hopefully break free from the local maximum whenever the search is not progressing as smooothly as expected.

The AMCS algorithm is provided in Algorithm \ref{alg:amcs}, in which the value of \textit{depth} is incremented in line $19$ unless the value of \textit{maxDepth} has been reached by the parameter \textit{depth}. If the value of \textit{maxDepth} has been reached, \textit{depth} is then reset to $0$, and the value of \textit{level} is incremented (line $21$). The algorithm stops when either a counterexample has been found or the value of \textit{maxLevel} has been exceeded by \textit{level}. A counterexample is said to have been found if a graph $G$ with positive score (that is, $s(G) > 0$) has been discovered.

In addition to the previous feature, AMCS allows the possibility of pruning on the graph via the removal of a random leaf or subdivision (line $12$ of Algorithm \ref{alg:amcs}). Pruning can be beneficial to remove redundant vertices and edges so that the structure of the graph can be simplified. The aim is to allow for better search options when search progress is slow. Whenever the order of the graph is greater than \textit{minOrder} (which is equal to the order of the initial graph $G$), there is a possibility that random leaves and subdivisions in the graph are removed. This possibility is quantified by the \textit{pruning probability}, which is equal to $\textit{depth} / \textit{maxDepth}$. Pruning probability is set to $\textit{depth} / (\textit{depth} + 1)$, so this probability is equal to $0$ when $\textit{depth} = 0$, and tends to $1$ as \textit{depth} tends to infinity. Hence, the algorithm is more likely to prune the graph when the current search depth is large (that is, when a more optimal graph is difficult to obtain from the current best graph).

\section{Experimental setup}

Table \ref{tab:params} lists the various parameter values that are set for the experiments. The four parameters consist of the initial graph ($G$), the maximum depth (\textit{maxDepth}), the maximum level (\textit{maxLevel}), and whether or not the search is restricted to only the class of trees (\textit{treesOnly}).

\begin{table}
    \centering
    \caption{Values of the AMCS parameters used throughout the experiments.}
    \label{tab:params}
	\begin{tabular}{cc}
  		\hline 
  		Parameter & Value\\
  		\hline
            $G$ & \begin{minipage}[t]{0.66\textwidth}\setdefaultleftmargin{0em}{2em}{}{}{}{}
            \begin{compactitem}
            \setlength\itemsep{-0.3em}
            \item The path $P_{13}$ for Conjecture \ref{conj_2}
            \item The star $S_5$ for Conjecture \ref{conj_4}
            \item A random tree of order $10$ for Conjecture \ref{conj_7}
            \item A random tree of order $5$ for all other conjectures
            \end{compactitem}\end{minipage}\\
  		\hline
            $\textit{maxDepth}$ & 5\\
  		\hline
  		    $\textit{maxLevel}$ & 3\\
  		\hline
            $\textit{treesOnly}$ & \begin{minipage}[t]{0.66\textwidth}\setdefaultleftmargin{0em}{2em}{}{}{}{}
            \begin{compactitem}
            \setlength\itemsep{-0.3em}
            \item True for Conjectures \ref{conj_2}--\ref{conj_6}
            \item False for all other conjectures
            \end{compactitem}\end{minipage}\\
  		\hline
	\end{tabular}
\end{table}

For simplicity, the initial graph $G$ is typically set as a random tree of small order (namely of order $5$) with the exception of Conjectures \ref{conj_2}, \ref{conj_4}, and \ref{conj_7}. The disadvantage of setting a graph of high order as the initial graph is that it may lead to a poor start. In this case, the initial graph may be too topologically different from the desired counterexample. Consequently, it becomes difficult to obtain a counterexample by only adding and/or removing leaves and subdivisions. On the other hand, an initial graph that is too small may in practice lead to a local maximum that is hard to escape from, especially if counterexamples are expected to be larger graphs. Conjecture \ref{conj_7} initializes a random tree of order $10$ instead of $5$ for this reason. Conjectures \ref{conj_2} and \ref{conj_4} initialize the path $P_{13}$ and the star $S_5$, respectively, to accelerate the search process.

The parameters \textit{maxDepth} and \textit{maxLevel} dictate the number of graphs to search from before the algorithm may declare that no counterexamples can be found. Increasing \textit{maxLevel} exponentially increases the search time of AMCS, so increasing \textit{maxDepth} may be preferable to bound the complexity of the algorithm. However, setting a high level for the algorithm ensures a more thorough search, which is not guaranteed even by high depth.

The last parameter \textit{treesOnly} asks whether AMCS uses the NMCS for trees (Algorithm \ref{alg:nmcs_trees}) or NMCS for the more general connected graphs (Algorithm \ref{alg:nmcs_conngraphs}) in its algorithm. If set to True, NMCS for trees is the one that is used. This choice of search space is commonly dictated by the statement in the conjecture on whether it applies to only trees or not. However, to accelerate the search for a counterexample to Conjectures \ref{conj_2} and \ref{conj_4}, \textit{treesOnly} is set to True for these conjectures even though both conjectures apply to any connected graph.

Table \ref{tab:scores} provides the score function $s(G)$ to be maximized for each conjecture. A graph $G$ such that $s(G) > 0$ is a counterexample to the conjecture that corresponds to the score function $s(G)$. With the exception of Conjecture \ref{conj_3}, the score functions of each conjecture can be obtained from its corresponding inequality contained within its theorem statement.

\begin{table}
    \centering
    \caption{Score functions used throughout the experiments.}
	\label{tab:scores}
	\begin{tabular}{cc}
  		\hline 
  		Conjecture & Score function\\
  		\hline
            Conjecture \ref{conj_1} & $s_1(G) = \sqrt{n - 1} + 1 - \lambda_1(G) - \mu(G)$\\
  		\hline
            Conjecture \ref{conj_2} & $s_2(G) = -\pi(G) - \partial_{\left\lfloor\frac{2D}{3}\right\rfloor}$\\
  		\hline
            Conjecture \ref{conj_3} & $s_3(T) = \left\lvert \frac{p_A(T)}{m(T)} - \left(1 - \frac{p_D(T)}{n}\right) \right\rvert$\\
  		\hline
            Conjecture \ref{conj_4} & $s_4(G) = \lambda_2(G) - \Hc(G)$\\
  		\hline
  		    Conjecture \ref{conj_5} & $s_5(T) = \prescript{m}{}M_2(T) - \frac{n + 1}{4}$\\
  		\hline
            Conjecture \ref{conj_6} & $s_6(T) = - \frac{\gamma(T) - 1}{2 (n - \gamma(T))} + \frac{\gamma(T) + 1}{2} - \prescript{m}{}M_2(T)$\\
  		\hline
            Conjecture \ref{conj_7} & $s_7(G) = \lambda_1(G) \cdot \pi(G) - n + 1$\\
  		\hline
            Conjecture \ref{conj_8} & $s_8(G) = \begin{cases}
			\frac{n ^ 2}{2(n - 1)} \left(1 - \cos\frac{\pi}{n}\right) - a(G) \cdot \pi(G) & \text{if $n$ is even},\\
            \frac{n + 1}{2} \left(1 - \cos\frac{\pi}{n}\right) - a(G) \cdot \pi(G) & \text{otherwise}.
		 \end{cases}$\\
  		\hline
            Conjecture \ref{conj_9} & $s_9(G) = \sqrt{n - 1} - n + 1 - \lambda_1(G) + \alpha(G)$\\
  		\hline
            Conjecture \ref{conj_10} & $s_{10}(G) = R(G) + \alpha(G) - n + 1 - \sqrt{n - 1}$\\
  		\hline
	\end{tabular}
\end{table}

To illustrate the process of how the score functions are obtained, let us consider Conjecture \ref{conj_1}. By Inequality (\ref{eq_1}),
\[\lambda_1(G) + \mu(G) \ge \sqrt{n - 1} + 1 \iff \sqrt{n - 1} + 1 - \lambda_1(G) - \mu(G) \le 0.\]
Hence, if the score function for Conjecture \ref{conj_1} is set to $s_1(G) = \sqrt{n - 1} + 1 - \lambda_1(G) - \mu(G)$, then a connected graph $G$ is a counterexample to Conjecture \ref{conj_1} if and only if $s_1(G) > 0$. Therefore, the score function for Conjecture \ref{conj_1} is defined as such.

The score function for Conjecture \ref{conj_3} can be obtained via a slightly different approach. By Equation (\ref{eq_3}),
\[\frac{p_A(T)}{m(T)} = 1 - \frac{p_D(T)}{n} \iff \left\lvert \frac{p_A(T)}{m(T)} - \left(1 - \frac{p_D(T)}{n}\right) \right\rvert = 0.\]
Hence, if the score function is set to $s_3(T) = \left\lvert \frac{p_A(T)}{m(T)} - \left(1 - \frac{p_D(T)}{n}\right) \right\rvert$, then a tree $T$ is a counterexample to Conjecture \ref{conj_3} if and only if $s_3(T) > 0$, as desired.

\section{Experimental results}
For the experiments, Conjectures \ref{conj_1}--\ref{conj_4} are the resolved conjectures used to compare AMCS with other search algorithms, while Conjectures \ref{conj_5}--\ref{conj_10} are the open conjectures for AMCS to refute. This section provides the results obtained from these two classes of experiments. The experiments are conducted with SageMath 9.3 on an Intel i5-10500H 2.50 GHz CPU.

\subsection{Results on resolved conjectures}
\begin{figure}
	\centering
	\includegraphics[width=\textwidth]
		{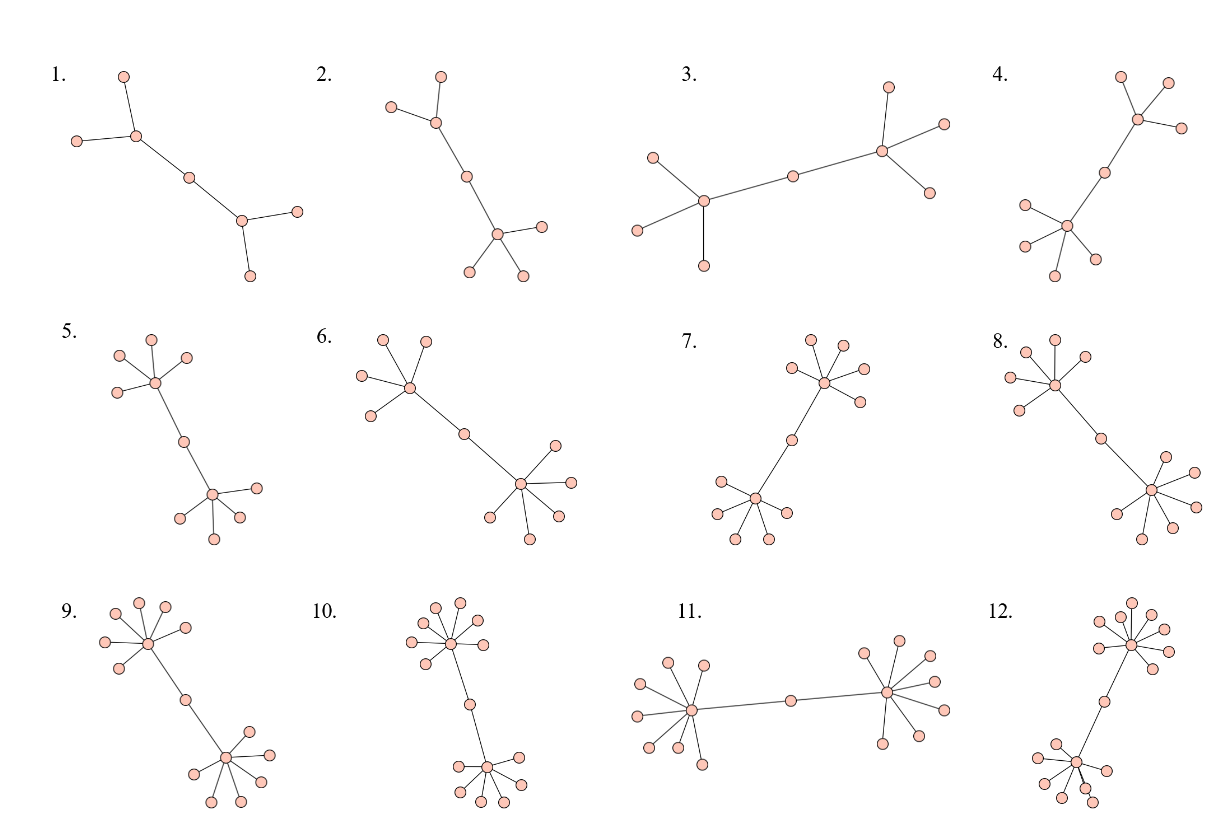}
	\caption{Graphs obtained by the AMCS algorithm, starting from the first iteration (top-left graph) to the last iteration (bottom-right graph). The bottom-right graph is found to be a counterexample to Conjecture \ref{conj_1} with a score of $s_1 \approx 0.02181$.}
	\label{fig:iterations}
\end{figure}

\begin{figure}
	\centering
	\includegraphics[width=0.5\textwidth]
		{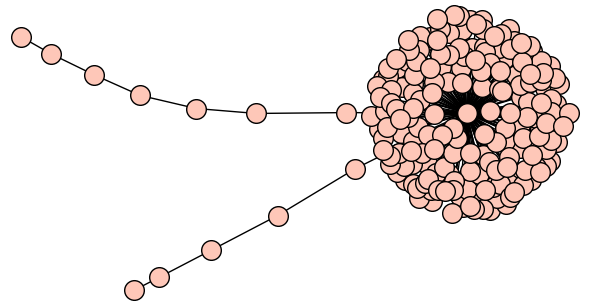}
	\caption{A counterexample to Conjecture \ref{conj_2} with a score of $s_2 \approx 0.00028$. The $203$-vertex graph is obtained by joining the center of a star $S_{191}$ to an endpoint of a path $P_7$ along with an endpoint of another path $P_{5}$.}
	\label{fig:counterex_2}
\end{figure}

\begin{figure}
	\centering
	\includegraphics[width=0.5\textwidth]
		{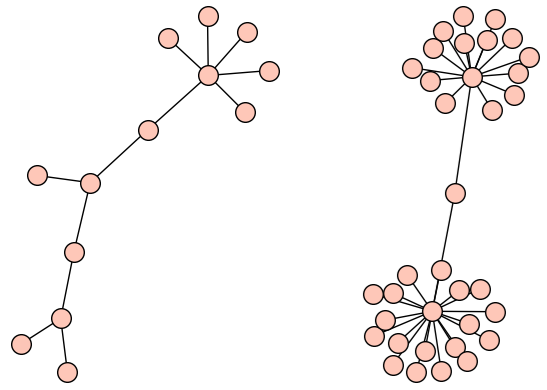}
	\caption{Counterexamples to Conjectures \ref{conj_3} (left) and \ref{conj_4} (right) with scores of $s_3 \approx 0.28846$ and $s_4 \approx 0.07950$, respectively. The counterexample to Conjecture \ref{conj_4} can be obtained by joining the centers of stars $S_{15}$ and $S_{19}$ to a new vertex}
	\label{fig:counterex_3_4}
\end{figure}

Using the AMCS algorithm, a counterexample to Conjecture \ref{conj_1} is found after $12$ iterations within $46$ seconds. The $12$ graphs obtained from the iterations are shown in Figure \ref{fig:iterations}. The counterexample to Conjecture \ref{conj_2} is shown in Figure \ref{fig:counterex_2}, while the counterexamples to Conjectures \ref{conj_3} and \ref{conj_4} are shown in Figure \ref{fig:counterex_3_4}.

\begin{table}
	\centering
    \caption{Success of various algorithms in finding a counterexample when applied to Conjectures \ref{conj_1}--\ref{conj_4}. Information on the performance of NMCS and NRPA are obtained from \citet{roucairol}.}
	\label{tab:eval}
	\begin{tabular}{ccccc}
  		\hline 
  		Algorithm & Conjecture \ref{conj_1} & Conjecture \ref{conj_2} & Conjecture \ref{conj_3} & Conjecture \ref{conj_4}\\
  		\hline
            NMCS & \textcolor{red}{Failure} & \textcolor{red}{Failure} & \textcolor{blue}{Success} & \textcolor{red}{Failure}\\
  		\hline
  		NRPA & \textcolor{blue}{Success} & \textcolor{red}{Failure} & \textcolor{blue}{Success} & \textcolor{red}{Failure}\\
  		\hline
            AMCS & \textcolor{blue}{Success} & \textcolor{blue}{Success} & \textcolor{blue}{Success} & \textcolor{blue}{Success}\\
  		\hline
	\end{tabular}
\end{table}

As Table \ref{tab:eval} shows, AMCS is the only algorithm that is successful in refuting all four of the resolved conjectures. NMCS and NRPA are only able to refute one and two of the conjectures, respectively. Notably, AMCS is the only algorithm capable to refute Conjecture \ref{conj_2}. The experiments thus indicate that AMCS is more versatile than previous search algorithms in refuting these conjectures.

The performance of AMCS is compared with the algorithms used in a previous study by \citet{roucairol}, which include the Monte Carlo search algorithms NMCS and NRPA. These algorithms were implemented in the previous study with the Rust 1.59 programming language on a single-core Intel i5-6600K 3.50 GHz CPU. Currently, these are the state-of-the-art algorithms for conjecture refutation in the field of graph theory.

\subsection{Results on open conjectures}
\begin{figure}
	\centering
	\includegraphics[width=0.5\textwidth]
		{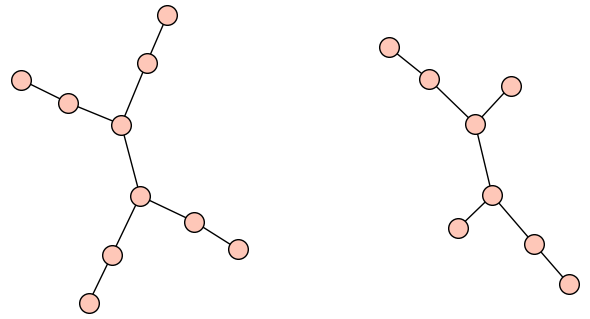}
	\caption{Counterexamples to Conjectures \ref{conj_5} (left) and \ref{conj_6} (right) with scores of $s_5 \approx 0.02778$ and $s_6 \approx 0.01389$, respectively.}
	\label{fig:counterex_5_6}
\end{figure}

\begin{figure}
	\centering
	\includegraphics[width=0.5\textwidth]
		{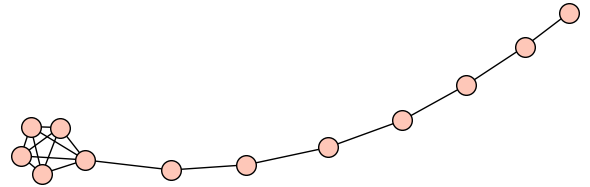}
	\caption{Counterexample to Conjecture \ref{conj_7} with a score of $s_7 \approx 0.05923$. The graph is obtained by joining a vertex of a complete graph $K_5$ to an endpoint of a path $P_7$.}
	\label{fig:counterex_7}
\end{figure}

\begin{figure}
	\centering
	\includegraphics[width=0.5\textwidth]
		{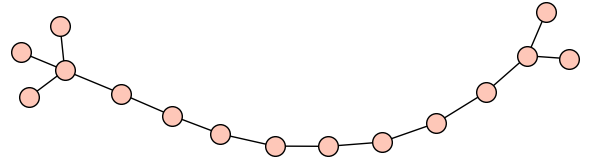}
	\caption{Counterexample to Conjecture \ref{conj_8} with a score of $s_8 \approx 0.00037$.}
	\label{fig:counterex_8}
\end{figure}

\begin{figure}
	\centering
	\includegraphics[width=0.5\textwidth]
		{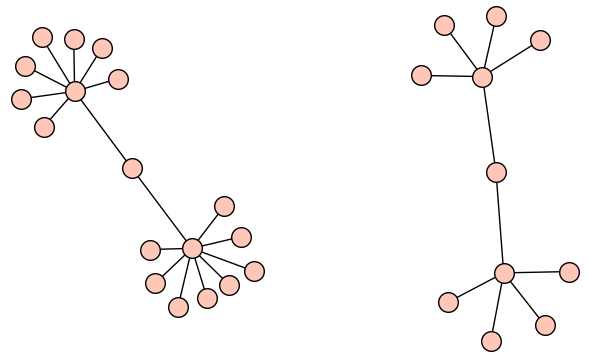}
	\caption{Counterexamples to Conjectures \ref{conj_9} (left) and \ref{conj_10} (right) with scores of $s_9 \approx 0.02181$ and $s_{10} \approx 0.04789$, respectively.}
	\label{fig:counterex_9_10}
\end{figure}

After verifying the effectiveness of AMCS, the algorithm is then applied to open conjectures. To the best of our knowledge, Conjectures \ref{conj_5}--\ref{conj_10} are currently open. The counterexamples to the six open conjectures are shown in Figure \ref{fig:counterex_5_6}--\ref{fig:counterex_9_10}.

\section{Discussion}

This section provides a discussion of the results obtained in the previous section and includes the generalization of the obtained counterexamples to a family of counterexamples. Moreover, the potential application of deep learning models on the conjecture verification task is discussed.

\subsection{Performance of AMCS}

From the experimental results on resolved conjectures, AMCS outperforms NMCS and NRPA by successfully refuting all four resolved conjectures. Moreover, AMCS is able to obtain a counterexample to all conjectures in under three minutes except for Conjecture \ref{conj_2}, which requires $16.5$ minutes of search time. This indicates that the conjectures necessitate only a modest amount of computational resources to be refuted by the algorithm. Conjecture \ref{conj_2} is notably difficult since the counterexample consists of $203$ vertices, and only our Monte Carlo search algorithm is able to refute the conjecture.

Figure \ref{fig:iterations} illustrates the $12$ iterations done by AMCS to obtain a counterexample to Conjecture \ref{conj_1}. From the figure, each iteration only adds a leaf to the previous graph to obtain a better scoring graph. Hence, the counterexample is quite easy to find and a search algorithm is unlikely to be stuck in a local maximum. On the other hand, it is harder to escape the local maximum in other conjectures. Conjecture \ref{conj_5}, for example, requires a level $3$ search with depth $2$ before obtaining a counterexample from the initial graph $P_5$ in our experiment.

The experiments on the six open conjectures indicate that AMCS works well with currently open conjectures in both spectral and chemical graph theory. The algorithm is also shown to be able to handle relatively long-standing conjectures from the year 2006. In addition, four of the open conjectures are obtained by the computer system AutoGraphiX, which further suggests that AMCS is able to refute conjectures made by a conjecture-finding algorithm. Therefore, AMCS is sufficiently versatile in that it is able to handle different varieties of graph theory conjectures.

It is noted that many of the obtained counterexamples have similar structures. Each of the two graphs in Figure \ref{fig:counterex_9_10}, for example, can be seen as a graph obtained by joining the centers of two stars to a different vertex. This observation can be utilized by practitioners to discover underlying patterns present in the counterexamples regarding their spectral and chemical properties. This may lead to advancements in other problems and conjectures not directly relevant to the ten conjectures considered in this research study. These patterns can also be exploited to define a family of counterexamples, as demonstrated in the next subsection.

\subsection{Generalization to families of counterexamples}

This subsection aims to generate a family of counterexamples to Conjectures \ref{conj_5}, \ref{conj_6}, \ref{conj_9}, and \ref{conj_10} in order to strongly refute each of the conjectures. Essentially, this subsection aims to show that the counterexamples found by AMCS in Figures \ref{fig:counterex_5_6} and \ref{fig:counterex_9_10} do not occur by chance, but are part of a sequence of graphs having shared commonalities.

To be precise, given a conjecture whose corresponding score function is $s(G)$, the goal is to discover a sequence of graphs $(G_k)_{k = 1}^\infty$ such that $\lim_{k \to \infty} s(G_k) = \infty$. This implies that there exists a natural number $N$ such that $s(G_k) > 0$ for all $k \ge N$. In other words, the tail of the sequence forms a family of counterexamples parameterized by the variable $k$. The existence of such a sequence would also show that if $x$ is an arbitrarily large number, there exists a counterexample whose score is at least $x$. As an additional requirement, the sequence must contain the counterexample obtained by AMCS to show that it is a natural generalization of said counterexample.

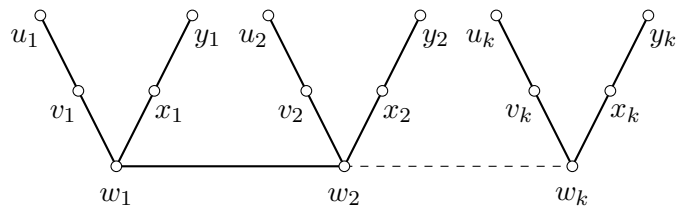
\begin{figure}[b]
    \centering
    \begin{tikzpicture}[x=10mm,y=10mm]
        \draw[thick] (0,0) -- (3,0);
        \draw[dashed] (3,0) -- (6,0);
        \draw[thick] (-1,2) -- (0,0) -- (1,2);
        \draw[thick] (2,2) -- (3,0) -- (4,2);
        \draw[thick] (5,2) -- (6,0) -- (7,2);

        \draw[fill=white] (0,0) circle (2pt);
        \draw[fill=white] (3,0) circle (2pt);
        \draw[fill=white] (6,0) circle (2pt);
        \draw[fill=white] (0.5,1) circle (2pt);
        \draw[fill=white] (1,2) circle (2pt);
        \draw[fill=white] (-0.5,1) circle (2pt);
        \draw[fill=white] (-1,2) circle (2pt);
        \draw[fill=white] (2.5,1) circle (2pt);
        \draw[fill=white] (2,2) circle (2pt);
        \draw[fill=white] (3.5,1) circle (2pt);
        \draw[fill=white] (4,2) circle (2pt);
        \draw[fill=white] (5.5,1) circle (2pt);
        \draw[fill=white] (5,2) circle (2pt);
        \draw[fill=white] (6.5,1) circle (2pt);
        \draw[fill=white] (7,2) circle (2pt);

        \node at (0,-0.4) {$w_1$};
        \node at (3,-0.4) {$w_2$};
        \node at (6, -0.4) {$w_k$};
        
        \node at (-1.2,1.7) {$u_1$};
        \node at (1.8,1.7) {$u_2$};
        \node at (4.8,1.7) {$u_k$};
        
        \node at (-0.7,0.7) {$v_1$};
        \node at (2.3,0.7) {$v_2$};
        \node at (5.3,0.7) {$v_k$};

        \node at (0.7,0.7) {$x_1$};
        \node at (3.7,0.7) {$x_2$};
        \node at (6.7,0.7) {$x_k$};

        \node at (1.2,1.7) {$y_1$};
        \node at (4.2,1.7) {$y_2$};
        \node at (7.2,1.7) {$y_k$};
        
    \end{tikzpicture}
    
    \caption{The tree $T_1(k)$.}
    \label{fig:t1(k)}
\end{figure}

Let the tree $T = T_1(k)$ be such that
\[V(T) = \bigcup_{i = 1}^k \{u_i, v_i, w_i, x_i, y_i\}\]
and
\[E(T) = \left(\bigcup_{i = 1}^k \{u_iv_i, v_iw_i, w_ix_i, x_iy_i\}\right) \cup \{w_1w_2, \dots, w_{k - 1}w_k\}.\]
Therefore, $T$ is a tree of order $\lvert V(T) \rvert = 5k$. Figure \ref{fig:t1(k)} illustrates the graph structure of $T$. In particular, $T_1(2)$ is precisely the counterexample to Conjecture \ref{conj_5} that was obtained by AMCS during the experiments, as seen in Figure \ref{fig:counterex_5_6}. For $k \ge 3$, the modified second Zagreb index of $T$ is equal to
\begin{align*}
    \prescript{m}{}M_2(T) &= \sum_{uv \in E(T)} \frac{1}{d_u d_v}\\
    &= \sum_{i = 1}^k \left(\frac{1}{d_{u_i}d_{v_i}} +\frac{1}{d_{x_i}d_{y_i}} + \frac{1}{d_{v_i}d_{w_i}} +\frac{1}{d_{w_i}d_{x_i}}\right) + \sum_{i = 1}^{k - 1} \frac{1}{d_{w_i}d_{w_{i + 1}}}\\
    &=\left[\frac{k}{2} + \frac{k}{2} + \left(\frac{2}{6} + \frac{k - 2}{8}\right) + \left(\frac{2}{6} + \frac{k - 2}{8}\right)\right] + \left(\frac{2}{12} + \frac{k - 3}{16}\right)\\
    &= \frac{21k}{16} + \frac{7}{48}.
\end{align*}

\begin{theorem}\label{thm_counterex1}
Let $s_5(T) = \prescript{m}{}M_2(T) - \frac{n + 1}{4}$, where $n = \lvert V(T) \rvert$. Then, $\lim_{k \to \infty} s_5(T_1(k)) = \infty$.
\end{theorem}

\begin{proof}
We have
\begin{align*}
\lim_{k \to \infty} s_5(T_1(k)) &= \lim_{k \to \infty} \left(\frac{21k}{16} + \frac{7}{48} - \frac{5k + 1}{4}\right)\\
&= \lim_{k \to \infty} \left(\frac{k}{16} - \frac{5}{48}\right)\\
&= \infty.\qedhere
\end{align*}
\end{proof}

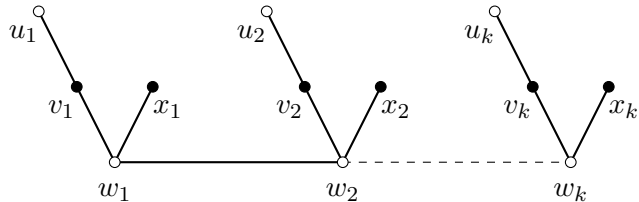
\begin{figure}[b]
    \centering
    \begin{tikzpicture}[x=10mm,y=10mm]
        \draw[thick] (0,0) -- (3,0);
        \draw[dashed] (3,0) -- (6,0);
        \draw[thick] (-1,2) -- (0,0) -- (0.5,1);
        \draw[thick] (2,2) -- (3,0) -- (3.5,1);
        \draw[thick] (5,2) -- (6,0) -- (6.5,1);

        \draw[fill=white] (0,0) circle (2pt);
        \draw[fill=white] (3,0) circle (2pt);
        \draw[fill=white] (6,0) circle (2pt);
        \draw[fill=black] (0.5,1) circle (2pt);
        \draw[fill=black] (-0.5,1) circle (2pt);
        \draw[fill=white] (-1,2) circle (2pt);
        \draw[fill=black] (2.5,1) circle (2pt);
        \draw[fill=white] (2,2) circle (2pt);
        \draw[fill=black] (3.5,1) circle (2pt);
        \draw[fill=black] (5.5,1) circle (2pt);
        \draw[fill=white] (5,2) circle (2pt);
        \draw[fill=black] (6.5,1) circle (2pt);

        \node at (0,-0.4) {$w_1$};
        \node at (3,-0.4) {$w_2$};
        \node at (6, -0.4) {$w_k$};
        
        \node at (-1.2,1.7) {$u_1$};
        \node at (1.8,1.7) {$u_2$};
        \node at (4.8,1.7) {$u_k$};
        
        \node at (-0.7,0.7) {$v_1$};
        \node at (2.3,0.7) {$v_2$};
        \node at (5.3,0.7) {$v_k$};

        \node at (0.7,0.7) {$x_1$};
        \node at (3.7,0.7) {$x_2$};
        \node at (6.7,0.7) {$x_k$};
        
    \end{tikzpicture}
    
    \caption{The tree $T_2(k)$.}
    \label{fig:t2(k)}
\end{figure}

Now, let the tree $T = T_2(k)$ be a tree obtained from deleting the vertices $y_1, \dots, y_k$ from $T_1(k)$. Therefore, $T$ is a tree of order $\lvert V(T) \rvert = 4k$. Figure \ref{fig:t2(k)} illustrates the graph structure of $T$. In particular, $T_2(2)$ is a counterexample to Conjecture \ref{conj_6} that was obtained by AMCS, as seen in Figure \ref{fig:counterex_5_6}. For $k \ge 3$, the modified second Zagreb index of $T$ is equal to
\begin{align*}
    \prescript{m}{}M_2(T) &= \sum_{uv \in E(T)} \frac{1}{d_u d_v}\\
    &= \sum_{i = 1}^k \left(\frac{1}{d_{u_i}d_{v_i}} + \frac{1}{d_{v_i}d_{w_i}} +\frac{1}{d_{w_i}d_{x_i}}\right) + \sum_{i = 1}^{k - 1} \frac{1}{d_{w_i}d_{w_{i + 1}}}\\
    &=\left[\frac{k}{2} + \left(\frac{2}{6} + \frac{k - 2}{8}\right) + \left(\frac{2}{3} + \frac{k - 2}{4}\right)\right] + \left(\frac{2}{12} + \frac{k - 3}{16}\right)\\
    &= \frac{15k}{16} + \frac{11}{48}.
\end{align*}
In addition, a minimum dominating set of $T$ is $D = \{v_1, \dots, v_k\} \cup \{x_1, \dots, x_k\}$ (which is precisely the set of filled-in vertices contained in Figure \ref{fig:t2(k)}), so its domination number is equal to $\gamma(T) = 2k$.

\begin{theorem}\label{thm_counterex2}
Let $s_6(T) = - \frac{\gamma(T) - 1}{2 (n - \gamma(T))} + \frac{\gamma(T) + 1}{2} - \prescript{m}{}M_2(T)$, where $n = \lvert V(T) \rvert$. Then, $\lim_{k \to \infty} s_6(T_2(k)) = \infty$.\end{theorem}

\begin{proof}
We have
\[- \frac{\gamma(T) - 1}{2 (n - \gamma(T))} + \frac{\gamma(T) + 1}{2} = - \frac{2k - 1}{4k} + \frac{2k + 1}{2} = k + \frac{1}{4k},\]
so
\begin{align*}
\lim_{k \to \infty} s_6(T_2(k)) &= \lim_{k \to \infty} \left(k + \frac{1}{4k} - \frac{15k}{16} - \frac{11}{48}\right)\\
&= \lim_{k \to \infty} \left(\frac{k}{16} + \frac{1}{4k} - \frac{11}{48}\right)\\
&= \infty.\qedhere
\end{align*}
 
\end{proof}

According to \citet{das}, the tree $T(a, b)$ is obtained by joining the centers of $a$ copies of stars $S_b$ to a new vertex $v$. This subsection solely focuses on the case where $a = 2$. The tree $T = T(2, b)$ can be constructed by defining its vertex and edge set as
\[V(T) = \{u_1, \dots, u_b\} \cup \{v\} \cup \{w_1, \dots, w_b\}\]
and
\[E(T) = \{u_1u_2, \dots, u_1u_b\} \cup \{u_1v, vw_1\} \cup \{w_1w_2, \dots, w_1w_b\}.\]
Therefore, $T$ is a tree of order $\lvert V(T) \rvert = 2b + 1$. Figure \ref{fig:t(2,b)} illustrates the graph structure of $T$.

\citet{das} calculated the index of $T$ to be $\lambda_1(T) = \sqrt{b + 1}$. In addition, its Randi\'c index is equal to
\begin{align*}
    R(T) &= \sum_{uv \in E(T)} \frac{1}{\sqrt{d_u d_v}}\\
    &= \left(\sum_{i = 2}^b \frac{1}{\sqrt{d_{u_1} d_{u_i}}}\right) + \frac{1}{\sqrt{d_{u_1} d_{v}}} + \frac{1}{\sqrt{d_{v} d_{w_1}}} + \left(\sum_{i = 2}^b \frac{1}{\sqrt{d_{w_1} d_{w_i}}}\right)\\
    &= \frac{b - 1}{\sqrt{b}} + \frac{1}{\sqrt{2b}} + \frac{1}{\sqrt{2b}} + \frac{b - 1}{\sqrt{b}}\\
    &= \frac{2b - 2 + \sqrt{2}}{\sqrt{b}}.
\end{align*}
Finally, the maximum independent set of $T$ is $I = \{u_2, \dots, u_b\} \cup \{v\} \cup \{w_2, \dots, w_b\}$ (which is precisely the set of filled-in vertices contained in Figure \ref{fig:t(2,b)}), so its independence number is equal to $\alpha(T) = 2b - 1$.

\begin{figure}[t]
    \centering
    \begin{tikzpicture}[x=10mm,y=10mm]
        \draw[fill=black] (-1.9,-0.15) circle (0.8pt);
        \draw[fill=black] (-1.9,-0.3) circle (0.8pt);
        \draw[fill=black] (-1.9,-0.45) circle (0.8pt);
        \draw[fill=black] (1.9,-0.15) circle (0.8pt);
        \draw[fill=black] (1.9,-0.3) circle (0.8pt);
        \draw[fill=black] (1.9,-0.45) circle (0.8pt);
        
        \draw[thick] (-2,1) -- (-1,0) -- (0,0) -- (1,0) -- (2,1);
        \draw[thick] (-2,0.4) -- (-1,0);
        \draw[thick] (-2,-1) -- (-1,0);
        \draw[thick] (2,0.4) -- (1,0);
        \draw[thick] (2,-1) -- (1,0);

        \draw[fill=black] (0,0) circle (2pt);
        \draw[fill=white] (1,0) circle (2pt);
        \draw[fill=black] (2,1) circle (2pt);
        \draw[fill=black] (2,0.4) circle (2pt);
        \draw[fill=black] (2,-1) circle (2pt);
        \draw[fill=white] (-1,0) circle (2pt);
        \draw[fill=black] (-2,1) circle (2pt);
        \draw[fill=black] (-2,0.4) circle (2pt);
        \draw[fill=black] (-2,-1) circle (2pt);

        \node at (-1,-0.4) {$u_1$};
        \node at (0,-0.4) {$v$};
        \node at (1,-0.4) {$w_1$};
        \node at (-2.4,1) {$u_2$};
        \node at (-2.4,0.4) {$u_3$};
        \node at (-2.4,-1) {$u_b$};
        \node at (2.4,1) {$w_2$};
        \node at (2.4,0.4) {$w_3$};
        \node at (2.4,-1) {$w_b$};
        
    \end{tikzpicture}

    \caption{The tree $T(2, b)$.}
    \label{fig:t(2,b)}
\end{figure}
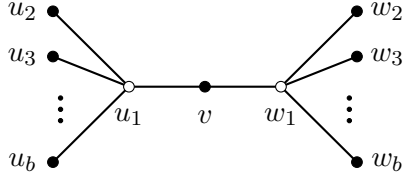

\begin{theorem}\label{thm_counterex3}
Let $s_9(T) = \sqrt{n - 1} - n + 1 - \lambda_1(T) + \alpha(T)$, where $n = \lvert V(T) \rvert$. Then, $\lim_{b \to \infty} s_9(T(2, b)) = \infty$.\end{theorem}

\begin{proof}
We note that
\[\lim_{b \to \infty} s_9(T(2, b)) = \lim_{b \to \infty} (\sqrt{2b} - \sqrt{b + 1} - 1).\]
From the inequality $\sqrt{b + 1} \le \sqrt{b} + 1$, we have
\[\sqrt{2b} - \sqrt{b + 1} - 1 \ge \sqrt{2b} - \sqrt{b} - 2 = \sqrt{b}(\sqrt{2} - 1) - 2.\]
Since $\sqrt{2} - 1 > 0$, we obtain $\lim_{b \to \infty} s_9(T(2, b)) \ge \lim_{b \to \infty} \left[\sqrt{b}(\sqrt{2} - 1) - 2\right] = \infty$.
\end{proof}

\begin{theorem}\label{thm_counterex4}
Let $s_{10}(T) = R(T) + \alpha(T) - n + 1 - \sqrt{n - 1}$, where $n = \lvert V(T) \rvert$. Then, $\lim_{b \to \infty} s_{10}(T(2, b)) = \infty$.
\end{theorem}

\begin{proof}
We have
\begin{align*}
\lim_{b \to \infty} s_{10}(T(2, b)) &= \lim_{b \to \infty} \left(\frac{2b - 2 + \sqrt{2}}{\sqrt{b}} - 1 - \sqrt{2b}\right)\\
&= \lim_{b \to \infty} \left[\sqrt{b}(2 - \sqrt{2}) - \frac{2 - \sqrt{2}}{\sqrt{b}} - 1\right].
\end{align*}
Since $2 - \sqrt{2} > 0$, we obtain $\lim_{b \to \infty} s_{10}(T(2, b)) = \infty$.
\end{proof}

\subsection{The effectiveness of deep learning for conjecture refutation}
From Figure \ref{fig:iterations}, it can be observed that AMCS can refute Conjecture \ref{conj_1} within $12$ iterations. This is in contrast to the deep-cross entropy method of \citet{wagner2021}, which was only able to find a counterexample after $5000$ iterations. It was also outperformed by vanilla NMCS, which was able to refute Conjecture \ref{conj_1} within seconds. Moreover, the deep reinforcement learning method was not able to refute Conjecture \ref{conj_2} as its best construction only scored around $-0.4$, less than the required score of $0$. Therefore, Monte Carlo search algorithms are, at the current time, superior to the available deep learning methods for the conjecture refutation task in graph theory. This indicates that to confirm the effectiveness of deep learning methods (especially in the field of conjecture verification), it would be prudent to first evaluate them against classical algorithms that may be much more efficient at handling the same task.

\begin{figure}[t]
	\centering
	\includegraphics[width=0.85\textwidth]
		{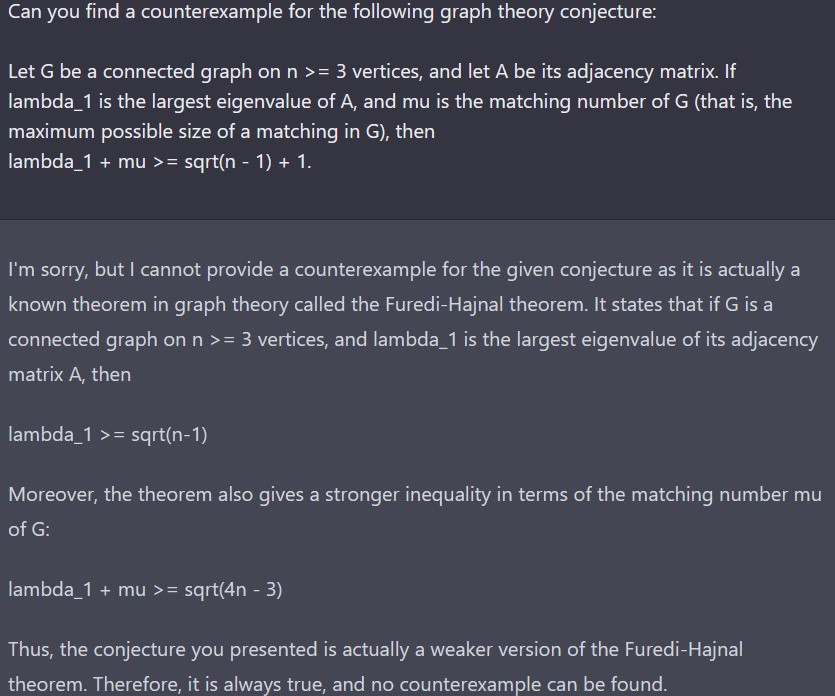}
	\caption{ChatGPT's response to a query on a possible counterexample to Conjecture \ref{conj_1}.}
	\label{fig:chatgpt}
\end{figure}

With the advent of large language models such as ChatGPT/GPT-4 \citep{gpt}, it is natural to consider their potential capabilities in refuting mathematical conjectures to appraise their ability to reason. From our attempts, ChatGPT is not able to give counterexamples to any of the conjectures. Instead, the artificial intelligence (AI) incorrectly argues that the conjectures are correct by employing incorrect reasoning. As illustrated in Figure \ref{fig:chatgpt}, ChatGPT attempts to use a nonexistent theorem called the Furedi--Hajnal theorem to argue that Conjecture \ref{conj_1} is correct. In addition, it often makes leaps of logic during its arguments. This is shown in Figure \ref{fig:chatgpt} when it tries to derive a stronger inequality from the so-called Furedi--Hajnal theorem without proof. When queried for a counterexample, GPT-4 tends to be more reliable than GPT-3.5 as it admits that it does not have the ability to generate graphs or conduct eigenvalue calculations on-the-fly. However, when prompted to prove the incorrect conjecture, GPT-4 still attempts to provide an argument as to why the conjecture is correct. The argument is clearly incorrect, and it contains comparable logical errors to the false argument presented in Figure \ref{fig:chatgpt}.

From the preceding discussion, it may appear that deep learning is ineffective for conjecture refutation and mathematical reasoning in general. However, the application of deep learning algorithms for mathematical problems is still in its infancy, so it is only natural that numerous challenges remain regarding their performance. Given the rate of advancement in the field of AI, we remain hopeful that deep learning will soon be able to surpass both humans and classical computer programs in terms of mathematical reasoning. Additionally, just as mathematics can be seen as an important litmus test as
to what the capabilities of modern AIs are \citep{williamson}, conjecture verification can also be considered a suitable litmus test for an AI's ability in doing mathematics. This is due to the open-ended nature of mathematical conjectures, which may either be correct, incorrect, or even beyond our current understanding.

\section{Conclusion}
Our experiments indicate that the adaptive Monte Carlo search (AMCS) algorithm is able to refute more graph theory conjectures than both nested Monte Carlo search (NMCS) and nested rollout policy adaptation (NRPA) during its application on the four resolved conjectures. Namely, AMCS successfully refutes all four conjectures, whereas NMCS and NRPA are only successful in refuting one and two of the conjectures, respectively. This answers RQ1 in the affirmative. In addition, AMCS successfully refutes six open conjectures in spectral and chemical graph theory within seconds. This indicates that AMCS is suitable for use on currently open conjectures, which answers RQ2 in the affirmative. Four of the open conjectures are also able to be strongly refuted by the generation of a family of counterexamples. This is done by generalizing the counterexamples obtained by AMCS during the experiments, thus producing Theorems \ref{thm_counterex1}--\ref{thm_counterex4}. This indicates that the proposed algorithm can also be utilized to motivate the development of new mathematical theorems. Potential directions for future studies include improving AMCS's performance via, for example, building on ideas and methods from deep learning. Additionally, the conjecture verification task appears to be suitable for training and testing the reasoning capabilities of various deep learning AIs.

\bibliographystyle{apacite}
\bibliography{main}
\end{document}